\documentclass[12pt]{amsart}

\usepackage{amssymb,amsmath,amsthm,hyperref,pgfplots,subcaption}

\newcommand{\A}{\mathbb{A}}
\newcommand{\CC}{\mathbb{C}}
\newcommand{\PP}{\mathbb{P}}

\newcommand{\BigFig}[1]{\parbox{12pt}{\Huge #1}}
\newcommand{\BigZero}{\BigFig{O}}

\DeclareMathOperator{\Char}{char}
\DeclareMathOperator{\ord}{ord}
\DeclareMathOperator{\Sing}{Sing}
\DeclareMathOperator{\rank}{rank}

\theoremstyle{plain}
\newtheorem{theorem}{Theorem}[section]
\newtheorem{lemma}[theorem]{Lemma}
\newtheorem{corollary}[theorem]{Corollary}
\newtheorem{proposition}[theorem]{Proposition}
\newtheorem*{main}{Theorem}

\theoremstyle{definition}
\newtheorem{definition}[theorem]{Definition}

\begin{document}

\mbox{}
\vspace{-1.1ex}
\title{Poncelet's theorem for conics in any position and any characteristic}
\author{Shin-Yao Jow}
\address{Department of Mathematics \\
National Tsing Hua University \\
Taiwan}
\email{\texttt{syjow@math.nthu.edu.tw}}
\author{Chia-Tz Liang}
\address{Department of Mathematics \\
Indiana University, Bloomington}
\email{\texttt{asctliang@gmail.com}}

\date{}

\begin{abstract}
Poncelet's theorem states that if there exists an $n$-sided polygon which is inscribed in a given conic $C$ and circumscribed about another conic $D$, then there are infinitely many such $n$-gons. Proofs of this theorem that we are aware of, including Poncelet's original proof and the celebrated modern proof by Griffiths and Harris, assume the two conics to be in general position (that is, not tangent or at least not osculating), or be defined over $\CC$, or both. Here we show that Poncelet's theorem holds for \emph{any} two conics $C$ and $D$ in the projective plane $\PP^2$ over an algebraically closed field $k$ of \emph{any} characteristic other than two. If $C$ and $D$ are osculating and $\Char k>2$, our result shows that there \emph{always} exist infinitely many polygons of $\Char k$ sides that are inscribed in $C$ and circumscribed about $D$. We also describe the situation in characteristic two in the appendix.
\end{abstract}


\maketitle

\subsection*{Notation} Let $\PP^2$ be the projective plane over an algebraically closed field~$k$. Let $p,q\in \PP^2$ be two distinct points, and let $C\subseteq \PP^2$ be a conic.
\begin{itemize}
  \item We write $L(p,q)$ for the line through $p$ and $q$.
  \item If $p\in C$, we write $T_p\,C$ for the line tangent to $C$ at $p$.
  \item If $\Char k\ne 2$, we write $P_q\,C$ for the polar line \cite[4.1]{fischer} of $C$ with respect to the pole $q$.
\end{itemize}
It is well-known that if $q\in C$ then $P_q\,C=T_q\,C$, while if $q\notin C$ then $P_q\,C$ is the line through the two points of tangency of the tangents to $C$ through $q$.

\section{Introduction}
In 1813, Poncelet \cite{poncelet} discovered the following remarkable result, which is now known as Poncelet's theorem.
\begin{center}
\parbox{.95\linewidth}{\emph{Let $C$ and $D$ be two (general) conics in the real projective plane. If there exists an $n$-sided polygon which is inscribed in $C$ and circumscribed about $D$, then there are infinitely many such polygons, and every point of $C$ is a vertex of one of them.}}
\end{center}
Poncelet first proved the result for two circles. Then he deduced the general case by projectively transforming $C$ and $D$ into a pair of circles, which cannot be done if $C$ and $D$ are osculating (that is, intersecting at a point with intersection multiplicity greater than two). Thus, strictly speaking, Poncelet's argument does not apply to osculating conics. We refer the readers to the excellent survey papers \cite{C1,C2} for a modern account of Poncelet's argument, as well as later development surrounding this theorem. Here we just want to mention the most recent approach, introduced by Griffiths and Harris in \cite{GH}, where they proved Poncelet's theorem for two nowhere tangent conics $C$ and $D$ in the complex projective plane, by considering the projective curve \[
 E=\{(c,d)\in C \times D \mid c\in T_d\,D \}. \]
The assumption that $C$ and $D$ are nowhere tangent is used to show that $E$ is an elliptic curve, which is crucial in their argument.

The purpose of this paper is to prove Poncelet's theorem for \emph{any} two conics $C$ and $D$ in the projective plane $\PP^2$ over an algebraically closed field $k$ of \emph{any} characteristic other than two. Given a point $c_1\in C$, we may construct a sequence of points $(c_i,d_i)\in C\times D$ for $i=1,2,\ldots$ recursively as follows: pick $d_1\in D$ such that $L(c_1,d_1)$ is tangent to $D$, and for $i\ge 2$ let $c_i\in C$ be the point such that $L(c_{i-1},d_{i-1})\cap C=\{c_{i-1},c_i\}$, and let $d_i\in D$ be the point such that $P_{c_i}D\cap D=\{d_{i-1},d_i\}$. We refer to this construction as a \emph{Poncelet process} with initial point $c_1$, and if $(c_1,d_1)=(c_{n+1},d_{n+1})$ then we say that the process \emph{stops after $n$ steps}. Figure~\ref{subfig:typical Poncelet} shows a typical Poncelet process, while Figure~\ref{subfig:Poncelet stop} shows a Poncelet process that stops after three steps.

\begin{figure}[h]
 \centering
 \subcaptionbox{A typical Poncelet process\label{subfig:typical Poncelet}}
  {  \begin{tikzpicture}[scale=.7,line cap=round,line join=round,x=0.4cm,y=0.4cm]
    \clip(-12.,-12.) rectangle (14.,11.);
        \draw [rotate around={0.:(0.,0.)},line width=0.4pt] (0.,0.) ellipse (2.2627416997969525cm and 1.6cm);
        \draw [rotate around={8.130102354155973:(1.,-1.)},line width=0.4pt] (1.,-1.) ellipse (4.595122457735075cm and 3.6214845576891825cm);
        \draw [line width=0.4pt] (0.,8.)-- (-9.960764084704117,-4.1993947278829);
        \draw [line width=0.4pt] (-9.960764084704117,-4.1993947278829)-- (11.64372251035786,-3.770336913206851);
        \draw [line width=0.4pt] (11.64372251035786,-3.770336913206851)-- (-1.5929698944334316,7.713271495294102);
        \begin{scriptsize}
            \draw (-0.6085375110409483,2.9960827095209117) node {$D$};
            \draw [fill=black] (0.,8.) circle (2.5pt);
            \draw (0,9) node {$c_1$};
            \draw (-7.2,3.8) node {$C$};
            \draw [fill=black] (-9.960764084704117,-4.1993947278829) circle (2.0pt);
            \draw (-10.7,-4.7) node {$c_2$};
            \draw [fill=black] (11.64372251035786,-3.770336913206851) circle (2.0pt);
            \draw (12.2,-4.5) node {$c_3$};
            \draw [fill=black] (-4.898979485566357,2.) circle (2.0pt);
            \draw (-3.7,1.8) node {$d_1$};
            \draw [fill=black] (0.15881467574809882,-3.998423307928204) circle (2.0pt);
            \draw (0.2,-3) node {$d_2$};
            \draw [fill=black] (4.384878809325144,2.5271365047748526) circle (2.0pt);
            \draw (3.8029909198780083,1.5) node {$d_3$};
            \draw [fill=black] (-1.5929698944334316,7.713271495294102) circle (2.0pt);
            \draw (-2,8.5) node {$c_4$};
        \end{scriptsize}
     \end{tikzpicture} }
  \subcaptionbox{A Poncelet process that stops after three steps\label{subfig:Poncelet stop}}
   { \begin{tikzpicture}[scale=.7,line cap=round,line join=round,x=1.0cm,y=1.0cm]
\clip(-4.5,-4.5) rectangle (4.5,4.5);
\draw [line width=0.4pt] (0.,0.) circle (4.cm);
\draw [line width=0.4pt] (-2.,0.) circle (1.5cm);
\draw [line width=0.4pt] (2.0780924877172398,3.417825567885695)-- (-3.91948610191114,0.7985165602073765);
\draw [line width=0.4pt] (-3.91948610191114,0.7985165602073765)-- (-2.3522410968808103,-3.235268431234871);
\draw [line width=0.4pt] (-2.3522410968808103,-3.235268431234871)-- (2.0780924877172398,3.417825567885695);
\begin{scriptsize}
\draw (-2.5180164604723565,3.6) node {$C$};
\draw [fill=black] (2.0780924877172398,3.417825567885695) circle (2.5pt);
\draw (2.5,3.8) node {$c_1=c_4$};
\draw (-2,-1.1) node {$D$};
\draw [fill=black] (-0.7514862664242724,-0.8313924807651301) circle (2.0pt);
\draw (-1.1,-0.5) node {$d_3$};
\draw [fill=black] (-2.6003371158438604,1.3746255298590504) circle (2.0pt);
\draw (-2.2,1) node {$d_1=d_4$};
\draw [fill=black] (-3.91948610191114,0.7985165602073765) circle (2.0pt);
\draw (-4.2,1.1) node {$c_2$};
\draw [fill=black] (-3.3981766177318664,-0.5432330490939209) circle (2.0pt);
\draw (-3,-0.38367095724361827) node {$d_2$};
\draw [fill=black] (-2.3522410968808103,-3.235268431234871) circle (2.0pt);
\draw (-2.5773319416814426,-3.5) node {$c_3$};
\end{scriptsize}
    \end{tikzpicture} }
  \caption{Poncelet process}\label{fig:Poncelet}
\end{figure}
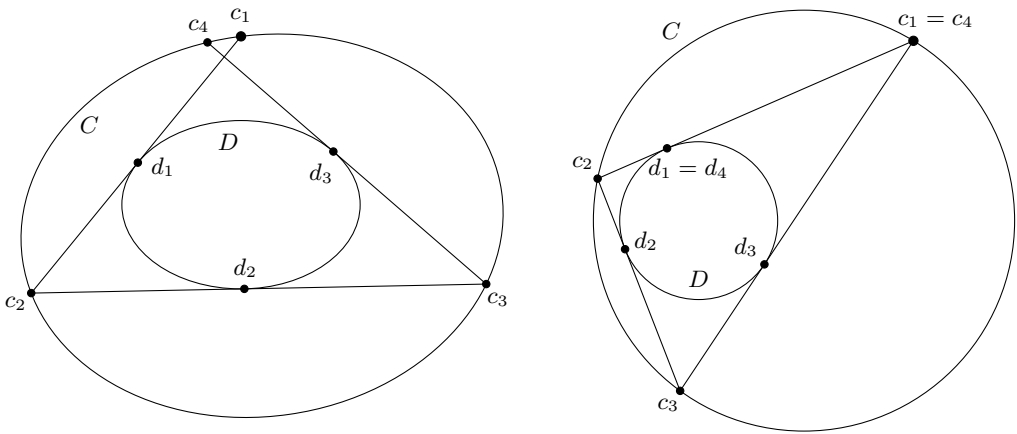

Using this terminology, our main result is the following

\begin{main}[\,$=$\,Theorem~\ref{thmm}]
Let $C$ and $D$ be two smooth conics in $\PP^2$ over an algebraically closed field $k$ with $\Char k\ne 2$. Let $T\subseteq C\cap D$ be the set of points where $C$ and $D$ are tangent. If there exist a point $c_1\in C\setminus T$ and a positive integer $n$ such that the Poncelet process with initial point $c_1$ stops after $n$ steps, then the Poncelet process with any initial point in $C$ stops after $n$ steps.

Moreover, when $C$ and $D$ are osculating,
    \begin{itemize}
        \item if $\Char k =0$, then the Poncelet process stops for no initial points in $C\setminus T$;
        \item if $\Char k >0$, then the Poncelet process stops for all initial points in $C$ after $\Char k$ steps.
    \end{itemize}
\end{main}

Note that a Poncelet process that starts at a tangent point of $C$ and $D$ stops after one step, which is certainly not the case for all other starting points. Thus it is necessary to exclude the tangent points of $C$ and $D$ as initial points in the assumption. The exclusion of characteristic two is because, if $\Char k=2$, then there exists a point $p\in \PP^2$ such that \emph{every} line through $p$ is tangent to $D$ (see Corollary \ref{cor: tangent of conic in char 2 }), so the Poncelet process cannot be defined if one of the point $c_i$ is equal to $p$.

Our approach is an expansion of that of Griffiths and Harris in \cite{GH}. We study how the geometric structure of the curve $E=\{(c,d)\in C \times D \mid c\in T_d\,D \}$ depends on the intersection multiplicities of $C$ and $D$ at their points of intersection. Over $\CC$ this has been done in \cite{leopold}, where complex analytic methods were used. Our approach, however, is purely algebraic, and applies in positive characteristic. As far as we know, the last part of our result, which says that the Poncelet process always stops for osculating conics in positive characteristic, seems to be new.

\subsection*{Acknowledgment} The authors gratefully acknowledge the support of MoST (Ministry of Science and Technology) in Taiwan.

\section{The algebraic curve underlying the Poncelet process}
In this section, let $C$ and $D$ be two smooth conics in $\PP^2$ over an algebraically closed field $k$ with $\Char k\ne 2$, and let \[
 E=\{(c,d)\in C \times D \mid c\in T_d\,D \}. \]
Griffiths and Harris \cite{GH} showed that $E$ is an elliptic curve if $C$ and $D$ are conics in general position over $k=\CC$, and used it to prove Poncelet's theorem under these assumptions. In this section, we will study the properties of $E$ without assuming $C$ and $D$ to be in general position nor defined over $\CC$, so that we can prove Poncelet's theorem without these assumptions later.

Our first proposition shows that $E$ is a (possibly reducible) projective curve whose singular points correspond to tangent points of $C$ and $D$.

\begin{proposition}\label{prop: E singular cond}
Let $C$ and $D$ be two smooth conics in $\PP^2$ over an algebraically closed field $k$ with $\Char k\ne 2$, and let \[
 E=\{(c,d)\in C \times D \mid c\in T_d\,D \}. \]
Then
\begin{enumerate}
  \item $E$ is a (possibly reducible) projective curve.
  \item Let $c\in C$ and $d\in D$ be points such that $p=(c,d)\in E$. Then $E$ is singular at $p$ if and only if the point $c$ and $d$ coincides and is a point at which the conics $C$ and $D$ are tangent to each other.
\end{enumerate}
\end{proposition}

\begin{proof}
Let
\begin{align*}
    &C\colon F(x,y,z)=\begin{pmatrix}x & y & z\end{pmatrix} A \begin{pmatrix}x\\y\\z\end{pmatrix},\\
    &D\colon G(x,y,z)=\begin{pmatrix}x & y & z\end{pmatrix} B \begin{pmatrix}x\\y\\z\end{pmatrix}
\end{align*}
be the quadratic forms that define the conics $C$ and $D$, respectively, where $A$ and $B$ are $3\times 3$ nondegenerate symmetric matrices. Then $E$ is the set of points \[
 ([x:y:z],[x':y':z'])\in \PP^2\times\PP^2 \]
defined by the following three equations:
        \begin{align*}
        &F(x,y,z)=\begin{pmatrix}x & y & z\end{pmatrix} A \begin{pmatrix}x\\y\\z\end{pmatrix}=0, \\
        &G(x',y',z')=\begin{pmatrix}x' & y' & z'\end{pmatrix} B \begin{pmatrix}x'\\y'\\z'\end{pmatrix}=0, \\
        &H(x,y,z,x',y',z')=\begin{pmatrix}x & y & z\end{pmatrix} B \begin{pmatrix}x'\\y'\\z'\end{pmatrix}=0.
        \end{align*}

If $c=[c_1:c_2:1]$ and $d=[d_1:d_2:1]$ are points such that $p=(c,d)\in E$, then the Jacobian matrix of the polynomials
    \begin{align*}
        &f(x,y)=F(x,y,1)=\begin{pmatrix}x & y & 1\end{pmatrix} A \begin{pmatrix}x\\y\\1\end{pmatrix},\\
        &g(x',y')=G(x',y',1)=\begin{pmatrix}x' & y' & 1\end{pmatrix} B \begin{pmatrix}x'\\y'\\1\end{pmatrix},\\
        &h(x,y,x',y')=H(x,y,1,x',y',1)=\begin{pmatrix}x & y & 1\end{pmatrix} B \begin{pmatrix}x'\\y'\\1\end{pmatrix}
    \end{align*}
    at $p$ is \[
    J=\begin{pmatrix}
                \dfrac{\partial f}{\partial x}(p) & \dfrac{\partial f}{\partial y}(p) & \dfrac{\partial f}{\partial x'}(p) & \dfrac{\partial f}{\partial y'}(p)\\
                \dfrac{\partial g}{\partial x}(p) & \dfrac{\partial g}{\partial y}(p) & \dfrac{\partial g}{\partial x'}(p) & \dfrac{\partial g}{\partial y'}(p)\\
                \dfrac{\partial h}{\partial x}(p) & \dfrac{\partial h}{\partial y}(p) & \dfrac{\partial h}{\partial x'}(p) & \dfrac{\partial h}{\partial y'}(p)
           \end{pmatrix}
          =\begin{pmatrix}
                \dfrac{\partial f}{\partial x}(p) & \dfrac{\partial f}{\partial y}(p) & 0 & 0\\
                0 & 0 & \dfrac{\partial g}{\partial x'}(p) & \dfrac{\partial g}{\partial y'}(p) \\
                \dfrac{\partial h}{\partial x}(p) & \dfrac{\partial h}{\partial y}(p) & \dfrac{\partial h}{\partial x'}(p) & \dfrac{\partial h}{\partial y'}(p)
           \end{pmatrix}. \]
    It suffices to show that $\rank(J)=3$ for general points $p$, and to find the points $p$ where $\rank(J)\le 2$ (these are the singular points of $E$). For any line $L\colon ax+by=e$ in $\A^2$, let $N(L)=[a:b]$ denote its ``normal vector''. Then we have
    \begin{align*}
        & N(T_c\,C)=\begin{bmatrix}
            \dfrac{\partial f}{\partial x}(p):\dfrac{\partial f}{\partial y}(p)
        \end{bmatrix},\\
        &N(P_d\,D)=\begin{bmatrix}
            \dfrac{\partial h}{\partial x}(p):\dfrac{\partial h}{\partial y}(p)
        \end{bmatrix}=
        \begin{bmatrix}
            \dfrac{\partial g}{\partial x'}(p):\dfrac{\partial g}{\partial y'}(p)
        \end{bmatrix}=N(T_d\,D),\\
        & N(P_c\,D)=\begin{bmatrix}
            \dfrac{\partial h}{\partial x'}(p):\dfrac{\partial h}{\partial y'}(p)
        \end{bmatrix}.
    \end{align*}
    Since $C$ and $D$ are smooth, $N(T_c\,C)$ and $N(T_d\,D)$ are nonzero, so $\rank(J)\ge 2$. Moreover,
    \begin{align*}
     \rank(J)=2 &\iff N(T_c\,C)=N(P_d\,D)=N(T_d\,D)=N(P_c\,D)\\
                &\iff T_c\,C=P_d\,D=T_d\,D=P_c\,D,
    \end{align*}
    where the second $\iff$ is because $T_c\,C$ and $T_d\,D$ have a common point $c$, $T_d\,D$ and $P_c\,D$ have a common point $d$, and $P_d\,D=T_d\,D$ since $d\in D$. Since $P_c\,D=P_d\,D$ if and only if $c=d$, it follows that $\rank(J)=2$ if and only if $c=d$ and $T_c\,C=T_d\,D$, that is, $C$ and $D$ are tangent at the point $c=d$.
\end{proof}

As observed in \cite{GH}, the Poncelet process can be viewed as performing the following two involutions $\sigma$ and $\tau$ on $E$ consecutively:

\begin{definition}\label{def: auto on E}
    For each point $(c,d)\in E\subseteq C\times D$, let $D\cap P_c\,D=\{d,d'\}$ and $C\cap T_d\,D=\{c,c'\}$, and let $\sigma$ and $\tau$ be the involutions on $E$ defined by \[
    \sigma(c,d)=(c,d')\text{ and }\tau(c,d)=(c',d).\]
\end{definition}

To study the Poncelet process for two tangent conics (in this case $E$ is singular), we will need the following result on the fixed points of the composition $\sigma \tau$.

\begin{proposition}\label{prop: singular condition for fixed point}
 A point $p\in E$ is a fixed point of the automorphism $\nu=\sigma \tau$ on $E$ if and only if $E$ is singular at $p$.
\end{proposition}

\begin{proof}
   Let us write $p=(c_1,d_1)$ and $\nu(p)=(c_2,d_2)$, where $c_1,c_2\in C$ and $d_1,d_2\in D$. It follows from the definition of $\nu$ that $D\cap P_{c_2}D=\{d_1,d_2\}$ and $C\cap T_{d_1}D=\{c_1,c_2\}$. So
    \begin{align*}
        c_1=c_2 &\iff T_{d_1}D \text{ is tangent to $C$ (at $c_1=c_2$)},\\
        d_1=d_2 &\iff c_2\in D.
    \end{align*}
    By Proposition~\ref{prop: E singular cond}, it is enough to show that
    \[
    \nu(p)=p\iff c_1=d_1,\text{  and }C\text{  and }D\text{ are tangent at }c_1=d_1.
    \]
    If $\nu(p)=p$, then $c_2\in T_{d_1}D$ and $c_2\in D$, so $c_2=d_1$ (since $d_1$ is the only intersection of $T_{d_1}D$ and $D$). Hence $c_1=c_2=d_1=d_2\in C\cap D$ and $T_{c_1}C=T_{d_1}D$, that is, $C$ and $D$ are tangent at the point $c_1=d_1$. The converse is obvious.
\end{proof}

We also need to understand the geometric structure of $E$, which turns out to depend on how $C$ and $D$ intersect each other.

\begin{definition}
    Let $C,D\subseteq \PP^2$ be projective plane curves. Let $p_1, \ldots, p_n$ be all the distinct intersection points of $C$ and $D$. Let $m_i=I_{p_i}(C,D)$ be the intersection multiplicity of $C$ and $D$ at $p_i$. We may assume that $m_1\ge \cdots \ge m_n$ after suitably renumbering $p_1,\ldots, p_n$. Then we say that the \emph{intersection type} of $C$ and $D$ is $(m_1,\ldots,m_n)$.
\end{definition}

By B\'ezout's theorem, $m_1+\cdots + m_n=(\deg C)(\deg D)$. So if $C$ and $D$ are conics, then there are five possible intersection types: $(1,1,1,1)$, $(2,1,1)$, $(2,2)$, $(3,1)$, and $(4)$. It was shown in \cite{GH} that if $C$ and $D$ are nowhere tangent conics, that is, of intersection type $(1,1,1,1)$, then $E$ is an elliptic curve. We include a proof here for the sake of completeness.

\begin{proposition}\label{cor: general position is elliptic curve}
  If $C$ and $D$ are nowhere tangent conics, then the curve $E$ defined in Proposition~\ref{prop: E singular cond} is an elliptic curve.
\end{proposition}

\begin{proof}
 Since $C$ and $D$ are not tangent, $E$ is a smooth projective (possibly reducible) curve by Proposition~\ref{prop: E singular cond}. On the other hand, $E\subseteq C\times D\cong \PP^1\times \PP^1$, and the two projection maps $E\to C$ and $E\to D$ are both finite of degree two. Hence $E$ is a smooth curve of bidegree $(2,2)$ in $\PP^1\times \PP^1$. If $E$ is reducible, say $E=E_1 \cup E_2$ with $E_i$ of bidegree $(a_i,b_i)$, then $E_1\cap E_2\ne \emptyset$ since $E_1\cdot E_2=a_1b_2+a_2b_1 > 0$, contradicting the smoothness of $E$. Hence $E$ is irreducible, and it follows from the adjunction formula that $E$ is of genus one.
\end{proof}

If $C$ and $D$ are tangent, we will find the defining equation of $E$ in $\PP^1\times \PP^1$ to understand its geometric structure. For this, we modify the approach in \cite{leopold}, replacing any complex analytic argument by algebraic one. The first step is to write the equations of $C$ and $D$ in a suitably chosen homogeneous coordinate system for $\PP^2$. We will use the following lemma, whose proof is straightforward and thus omitted.

\begin{lemma}\label{lem: reduce C}
    Let $C$ be a smooth conic in $\PP^2$ over $k$ with $\Char k\ne 2$. If the point $p=[0:0:1]\in C$, and the tangent of $C$ at $p$ is given by $y=0$, then the defining equation of $C$ is of the form \[
            C\colon x^2+txy+ay^2-byz=0 \]
    for some $t,a,b\in k$ with $b\ne 0$.
\end{lemma}

\begin{proposition}\label{prop: reduce C D and intersection type}
    Let $C$ and $D$ be two smooth conics in $\PP^2$ over $k$ with $\Char k\ne 2$ that are tangent to each other at a point $p$. Then there exists a suitable homogeneous coordinate system for $\PP^2$ such that $p=[0:0:1]$, and the defining equations of $C$ and $D$ are of the form
    \begin{align*}
        &C\colon F(x,y,z)=x^2+txy+ay^2-byz=0,\\
        &D\colon G(x,y,z)=x^2-yz=0
    \end{align*}
    for some $t,a,b\in k$ with $b\ne 0$. Moreover, if we write $\Delta=t^2-4a(1-b)$, then the intersection type of $C$ and $D$ is
    \begin{itemize}
    \item $(2,1,1) \iff b\ne 1$ and $\Delta\ne 0$;
    \item $(2,2) \iff b\ne 1$ and $\Delta=0$;
    \item $(3,1) \iff b=1$ and $t\ne 0$;
    \item $(4) \iff b=1$, $t=0$, and $a\ne 0$.
    \end{itemize}
\end{proposition}

\begin{proof}
    Let $L$ be the common tangent of $C$ and $D$ at $p$. Choose a homogeneous coordinate system $[x':y':z']$ for $\PP^2$ such that $p=[0:0:1]$ and $L\colon y'=0$. By Lemma~\ref{lem: reduce C}, the defining equations of $C$ and $D$ are of the form
    \begin{align*}
        &C\colon F(x',y',z')=x'^2+t_1x'y'+a_1y'^2-b_1y'z'=0,\\
        &D\colon G(x',y',z')=x'^2+t_2x'y'+a_2y'^2-b_2y'z'=0
    \end{align*}
    for some $t_i,a_i,b_i\in k$ with $b_i\ne 0$. Then in the homogeneous coordinate system \[
     [x:y:z]=[x'+\frac{t_2}{2}y' : y' : (-a_2+\frac{t_2^2}{4})y'+b_2z'], \]
    the defining equation of $D$ becomes \[
           D\colon G(x,y,z)=x^2-yz=0. \]
    Note that in the coordinate system $[x:y:z]$, we still have $p=[0:0:1]$ and $L\colon y=0$, so by Lemma~\ref{lem: reduce C}, the defining equation of $C$ is of the form \[
        C\colon F(x,y,z)=x^2+txy+ay^2-byz=0 \]
    for some $t,a,b\in k$ with $b\ne 0$.

    Let
    \begin{align*}
        f(x,y)&=F(x,y,1)=x^2+txy+ay^2-by,\\
        g(x,y)&=G(x,y,1)=x^2-y.
    \end{align*}
    Since $D$ is smooth and $x$ is a parameter for $D$, \[
        I_p(C,D)=I_{(0,0)}(f,g)=\ord_{(0,0)}(f|_D)=\ord_{x=0}((1-b)x^2+tx^3+ax^4). \]
    So the intersection type of $C$ and $D$ is
    \begin{itemize}
        \item $(2,1,1)$ or $(2,2)\iff I_p(C,D)=2 \iff b\ne 1$;
        \item $(3,1)\iff I_p(C,D)=3 \iff b=1$ and $t\ne 0$;
        \item $(4)\iff I_p(C,D)=4 \iff b=1$, $t=0$, and $a\ne 0$.
    \end{itemize}
    If $I_p(C,D)=2$, the intersection points of $C$ and $D$ other than $p$ can be computed by plugging $yz=x^2$ into the equation of $C$, which gives \[
        (1-b)x^2+txy+ay^2=0.\]
    The discriminant of this quadratic equation is $\Delta=t^2-4a(1-b)$, so the intersection type of $C$ and $D$ is
    \begin{itemize}
     \item $(2,1,1)\iff b\ne 1$ and $\Delta\ne 0$;
     \item $(2,2) \iff b\ne 1$ and $\Delta=0$.
    \end{itemize}
\end{proof}

We can now give a complete description of the projective curve $E$ defined in Proposition~\ref{prop: E singular cond} when $C$ and $D$ are tangent.

\begin{proposition}\label{prop: shape of E}
  Let $C$ and $D$ be two smooth conics in $\PP^2$ over an algebraically closed field $k$ with $\Char k\ne 2$, and let \[
   E=\{(c,d)\in C \times D \mid c\in T_d\,D \}. \]
  If $C$ and $D$ are tangent, then $E$ is
 \begin{itemize}
   \item an irreducible rational curve with a node and smooth elsewhere if the intersection type of $C$ and $D$ is $(2,1,1)$;
   \item two smooth rational curves intersecting at two points transversally if the intersection type of $C$ and $D$ is $(2,2)$;
   \item an irreducible rational curve with an ordinary cusp and smooth elsewhere if the intersection type of $C$ and $D$ is $(3,1)$;
   \item two smooth rational curves intersecting at one point with multiplicity two if the intersection type of $C$ and $D$ is $(4)$.
 \end{itemize}
Moreover, in the two cases where $E$ is reducible, the involutions $\sigma$ and $\tau$ in Definition~\ref{def: auto on E} both interchange the two irreducible components of $E$.
\end{proposition}

\begin{proof}
By Proposition~\ref{prop: reduce C D and intersection type}, we can choose a homogeneous coordinate system for $\PP^2$ such that $C$ and $D$ are tangent at $[0:0:1]$, and the defining equations of $C$ and $D$ are of the form
    \begin{align*}
        &C\colon F(x,y,z)=x^2+txy+ay^2-byz=0,\\
        &D\colon G(x,y,z)=x^2-yz=0
    \end{align*}
for some $t,a,b\in k$ with $b\ne 0$. Parametrizing $C$ with the parameter $u=y/x$ gives an isomorphism \[
 \PP^1\xrightarrow{\ \cong\ } C,\quad u\in k\cup \{\infty\}\longmapsto [bu:bu^2:1+tu+au^2]\in C. \]
Similarly, parametrizing $D$ with the parameter $v=y/x$ gives an isomorphism \[
 \PP^1\xrightarrow{\ \cong\ } D,\quad v\in k\cup \{\infty\}\longmapsto [v:v^2:1]\in D. \]
Since the tangent of $D$ at a point $d=[v:v^2:1]\in D$ is given by \[
         T_d\,D\colon 2vx-y-v^2z=0, \]
we find that \[
        E=\{(c,d)\in C \times D \mid c\in T_d\,D \}
        \cong\{(u,v)\in\PP^1\times\PP^1\mid H(u,v)=0\}, \]
where  \[
   H(u,v)=bu^2-2buv+(1+tu+au^2)v^2. \]

We have
\begin{align*}
   \text{$E$ is reducible} &\iff \text{$H$ is reducible in $k[u,v]=k[u][v]$}\\
                &\iff \text{$H$ is reducible in $k(u)[v]$},
\end{align*}
where the second $\iff$ is due to Gauss's lemma. Since $H$ is a quadratic polynomial in $k(u)[v]$, it is reducible if and only if its discriminant \[
 (2bu)^2-4bu^2(1+tu+au^2)=-4bu^2(1-b+tu+au^2) \]
is a perfect square. Hence
\begin{align*}
  E\text{ is reducible} &\iff 1-b+tu+au^2 \text{ is a square in } k[u] \\
    &\iff \Delta=t^2-4a(1-b)=0.
\end{align*}
Thus $E$ is reducible if the intersection type of $C$ and $D$ is $(2,2)$ or $(4)$ by Proposition~\ref{prop: reduce C D and intersection type}. Moreover, in this case, the irreducible factorization of the polynomial $H$ is of the form $H(u,v)=H_1(u,v)\cdot H_2(u,v)$, where both $H_1$ and $H_2$ are polynomials of bidegree~$(1,1)$ in $(u,v)$. It follows that the curves $E_1$ and $E_2$ defined respectively by $H_1$ and $H_2$ are both smooth rational curves, and they are the irreducible components of $E$. If the involution $\sigma$ in Definition~\ref{def: auto on E} sends a point $(c,d)\in E$ to $(c,d')$, then $d$ and $d'$ are exactly the solutions to the quadratic equation in $v$ given by \[
    H(c,v)=H_1(c,v)\cdot H_2(c,v)=0. \]
So if $(c,d) \in E_1$, then $d$ is the solution to $H_1(c,v)=0$, and hence $d'$ is the solution to $H_2(c,v)=0$, that is, $(c,d') \in E_2$. Thus $\sigma$ interchanges $E_1$ and $E_2$. A similar argument shows that $\tau$ also interchanges $E_1$ and $E_2$.

If the intersection type of $C$ and $D$ is $(2,2)$, then $E_1$ and $E_2$ intersect at two points by Proposition~\ref{prop: E singular cond}, and the intersection is transversal since the intersection number $E_1\cdot E_2=2$. If the intersection type of $C$ and $D$ is $(4)$, then $E_1$ and $E_2$ intersect at only one point by Proposition~\ref{prop: E singular cond}, and the intersection multiplicity at that point is equal to the intersection number $E_1\cdot E_2=2$.

Suppose that the intersection type of $C$ and $D$ is $(2,1,1)$ or $(3,1)$. Then $E$ is an irreducible curve of bidegree~$(2,2)$ in $\PP^1\times \PP^1$, so it is of arithmetic genus one by the adjunction formula. By Proposition~\ref{prop: E singular cond}, $E$ has only one singularity, which occurs at $(u,v)=(0,0)$. The singularity type depends on whether the quadratic part $bu^2-2buv+v^2$ of the defining polynomial $H(u,v)$ of $E$ is a perfect square: it is a node if $bu^2-2buv+v^2$ is not a perfect square, whereas if $bu^2-2buv+v^2$ is a perfect square then it is an ordinary cusp. Hence the singularity is
\begin{itemize}
  \item a node if $b\ne 1$, that is, if the intersection type of $C$ and $D$ is $(2,1,1)$;
  \item an ordinary cusp if $b=1$, that is, if the intersection type of $C$ and $D$ is $(3,1)$.
\end{itemize}
In both cases $E$ has geometric genus zero, so $E$ is rational.
\end{proof}

\section{Poncelet's theorem for any conics}
If two plane conics $C$ and $D$ are tangent, then the Poncelet process that starts at a tangent point of $C$ and $D$ stops after one step, which is certainly not the case for all other starting points. So the following version of Poncelet's theorem for possibly tangent conics is the best one could hope for.

\begin{theorem}\label{thmm}
Let $C$ and $D$ be two smooth conics in $\PP^2$ over an algebraically closed field $k$ with $\Char k\ne 2$. Let $T\subseteq C\cap D$ be the set of points where $C$ and $D$ are tangent. If there exist a point $c_1\in C\setminus T$ and a positive integer $n$ such that the Poncelet process with initial point $c_1$ stops after $n$ steps, then the Poncelet process with any initial point in $C$ stops after $n$ steps.

Moreover, in the cases where the intersection type of $C$ and $D$ are $(3,1)$ or $(4)$,
    \begin{itemize}
        \item if $\Char k =0$, then the Poncelet process stops for no initial points in $C\setminus T$;
        \item if $\Char k >0$, then the Poncelet process stops for all initial points in $C$ after $\Char k$ steps.
    \end{itemize}
\end{theorem}

\begin{proof}
Let $E=\{(c,d)\in C \times D \mid c\in T_d\,D \}$ be the projective curve defined in Proposition~\ref{prop: E singular cond}, where it is also shown that the singular locus $\Sing(E)$ of $E$ is \[
 \Sing(E)=\{(c,d)\in C\times D\mid c=d\in T\}. \]
Let $\sigma$ and $\tau$ be the involutions on $E$ in Definition~\ref{def: auto on E}, and let $\nu=\sigma \tau$. Then a sequence of points $(c_1,d_1),(c_2,d_2),\ldots\in C\times D$ is produced by the Poncelet process if and only if $(c_i,d_i)\in E$ and $\nu(c_i,d_i)=(c_{i+1},d_{i+1})$ for all $i\ge 1$. So what we want to show is equivalent to the statement that if there exists a positive integer~$n$ such that $\nu^n$ has a fixed point in $E\setminus \Sing(E)$, then $\nu^n$ is the identity map on $E$. Since the geometric structure of $E$ depends on the intersection type of $C$ and $D$, we divide the proof into several cases accordingly.

\textbf{The intersection type of $C$ and $D$ is $(1,1,1,1)$.} In this case, $E$ is an elliptic curve by Proposition~\ref{cor: general position is elliptic curve}. This is the case treated by Griffiths and Harris in \cite{GH}. They wrote their proof over $\CC$, but we will adopt their method and point out that it actually works in positive characteristic too. The key fact, which Griffiths and Harris only proved over $\CC$ but is actually true in general, is that if $E$ is an elliptic curve and $\alpha\colon E\to E$ is an involution with a fixed point, then there exists a point $a\in E$ such that $\alpha(p)=-p+a$ for all $p\in E$, where the addition here denotes the elliptic curve group law. Since $\sigma$ and $\tau$ are both involutions on $E$ with fixed points, there thus exist $a,b\in E$ such that $\sigma(p)=-p+a$ and $\tau(p)=-p+b$ for all $p\in E$. Hence \[
 \nu(p)=(\sigma \tau)(p)=-(-p+b)+a=p+(a-b), \]
so $\nu^n(p)=p+n(a-b)$ for all $p\in E$. If $\nu^n$ has a fixed point, then $n(a-b)=0$, and hence $\nu^n$ is the identity map.

\textbf{The intersection type of $C$ and $D$ is $(2,1,1)$.} In this case, $E$ is an irreducible rational curve with a node $q$ and smooth elsewhere by Proposition~\ref{prop: shape of E}. Let $\phi\colon \PP^1 \to E$ be the normalization of $E$. Then $\phi^{-1}(q)=\{p_1,p_2\}$ consists of two distinct points, and \[
\phi|_{U}\colon U=\PP^1\setminus \{p_1,p_2\} \longrightarrow E\setminus \{q\}=V \]
is an isomorphism. Let $\tilde{\nu}$ be the automorphism of $\PP^1$ induced by $\nu$. Since $q$ is the only fixed point of $\nu$ by Proposition \ref{prop: singular condition for fixed point}, $\tilde{\nu}$ has no fixed points in $U$. Since any automorphism of $\PP^1$ has a fixed point, $p_1$ and $p_2$ must both be fixed points of $\tilde{\nu}$. If there exists a positive integer~$n$ such that $\nu^n$ has a fixed point in $V$, then $\tilde{\nu}^n$ has a fixed point in $U$. Then $\tilde{\nu}^n$ is an automorphism of $\PP^1$ with at least three fixed points, so it is the identity map.

\textbf{The intersection type of $C$ and $D$ is $(3,1)$.} In this case, $E$ is an irreducible rational curve with an ordinary cusp $q$ and smooth elsewhere by Proposition~\ref{prop: shape of E}. Let $\phi\colon \PP^1 \to E$ be the normalization of $E$. Then $\phi^{-1}(q)=\{p\}$ consists of a single point, and \[
\phi|_{U}\colon U=\PP^1\setminus \{p\} \longrightarrow E\setminus \{q\}=V \]
is an isomorphism. Let $\tilde{\nu}$ be the automorphism of $\PP^1$ induced by $\nu$. Since $q$ is the only fixed point of $\nu$ by Proposition \ref{prop: singular condition for fixed point}, $\tilde{\nu}$ has no fixed points in $U$. Hence $p$ is the only fixed point of $\tilde{\nu}$. If we choose an affine coordinate $z$ on $\PP^1$ such that the point $p$ corresponds to $z=\infty$, then $\tilde{\nu}$ is given by $\tilde{\nu}(z)=z+b$ for some nonzero $b\in k$, so $\tilde{\nu}^n(z)=z+nb$ for all $z\in \PP^1$. If $\Char k =0$, then for any positive integer $n$ we have $nb\ne 0$, which implies that $\tilde{\nu}^n$ has no fixed points in $U$, and hence $\nu^n$ has no fixed points in $V$. On the other hand, if $\Char k >0$, then for $n=\Char k$ we have $nb=0$, so $\nu^n$ is the identity map.

\textbf{The intersection type of $C$ and $D$ is $(2,2)$.} In this case, by Proposition~\ref{prop: shape of E}, $E$ has two irreducible components $E_1$ and $E_2$, both isomorphic to $\PP^1$, which intersect transversally at two points $p_1$ and $p_2$. Since both $\sigma$ and $\tau$ interchange the two components, $\nu=\sigma\tau$ restricts to an automorphism on $E_i$ for $i=1,2$. By Proposition~\ref{prop: singular condition for fixed point}, $p_1$ and $p_2$ are fixed points of $\nu$. Suppose that there is a positive integer $n$ such that $\nu^n$ has a fixed point $q$ in $E\setminus \{p_1,p_2\}$. Without loss of generality, we may assume that $q\in E_1$. Then $\nu^n|_{E_1}$ is an automorphism of $E_1\cong \PP^1$ with at least three fixed points, so $\nu^n|_{E_1}$ is the identity map. Then $(\nu^n)^{-1}=(\tau\sigma)^n$ also restricts to the identity map on $E_1$, that is, $(\tau\sigma)^n(p)=p$ for all $p\in E_1$. It follows that, for any point $p\in E_2$, we have $(\tau\sigma)^n\bigl(\tau(p)\bigr)=\tau(p)$, and applying $\tau$ on both sides gives $(\sigma\tau)^n(p)=p$, that is, $\nu^n|_{E_2}$ is the identity map. Hence $\nu^n$ is the identity map on $E$.

\textbf{The intersection type of $C$ and $D$ is $(4)$.} In this case, by Proposition~\ref{prop: shape of E}, $E$ has two irreducible components $E_1$ and $E_2$, both isomorphic to $\PP^1$, which intersect at only one point $p$ with multiplicity two. Since both $\sigma$ and $\tau$ interchange the two components, $\nu=\sigma\tau$ restricts to an automorphism on $E_i$ for $i=1,2$. By Proposition~\ref{prop: singular condition for fixed point}, $p$ is the only fixed point of $\nu$ on $E$. For each $i=1,2$, let $z_i$ be an affine coordinate for $E_i \cong \PP^1$ such that the point $p$ corresponds to $z_i=\infty$. Then there exist nonzero $b_i\in k$ such that $\nu|_{E_i}(z_i)=z_i+b_i$, so $\nu^n|_{E_i}(z_i)=z_i+nb_i$ for all $z_i\in E_i$. If $\Char k =0$, then for any positive integer $n$ we have $nb_i\ne 0$, so $\nu^n$ has no fixed points in $E\setminus \{p\}$. On the other hand, if $\Char k >0$, then for $n=\Char k$ we have $nb_i=0$, so $\nu^n$ is the identity map.
\end{proof}

\appendix

\section{Quadrics in characteristic~$2$}
In this appendix, we present the structure of quadrics over an algebraically closed field~$k$ of characteristic~$2$ (see also \cite[Theorem~3.5, Corollary~3.6]{BM} where $k$ is not assumed to be algebraically closed). The purpose is to show that, in characteristic~$2$, one cannot define a meaningful Poncelet process for conics due to Corollary \ref{cor: tangent of conic in char 2 }.

\begin{proposition}
    Let $q\colon V\to k$ be a quadratic form on a  finite-dimensional $k$-vector space $V$. Then under a suitable coordinate system on $V$, $q$ is represented by a quadratic polynomial of the form
    \begin{align*}
        x_1y_1+x_2y_2+\dots+x_ly_l \ \text{ or }\  x_1y_1+x_2y_2+\dots+x_ly_l+x_{l+1}^2.
    \end{align*}
\end{proposition}

\begin{proof}
    Let \[
    B_q\colon V\times V \longrightarrow k, \quad (u,v)\longmapsto q(u+v)-q(u)-q(v) \]
    be the bilinear form defined by $q$. Note that, if $\{v_1, \dots, v_n\}$ is a  basis for $V$, and if we write \[
      q(\sum^n_{i=1}x_iv_i)=\sum^n_{i=1}a_ix_i^2+\sum_{1\le i<j\le n}a_{ij}x_ix_j,\]
    then the coefficient of $x_ix_j$ is $B_q(v_i,v_j)$ because \[
            B_q(v_i,v_j)=q(v_i+v_j)-q(v_i)-q(v_j)=a_i+a_j+a_{ij}-a_i-a_j=a_{ij}. \]
    Since $B_q(v,v)=q(v+v)-q(v)-q(v)=2q(v)=0$, $B_q$ is an alternating bilinear form, so there exists a basis of $V$ with respect to which $B_q$ is represented by a block-diagonal matrix of the form \[
       \begin{pmatrix}
             0 & 1 \\
             1 & 0 & & & & & & \BigZero \\
             & & \ddots  \\
             & & &0 & 1\\
             & & &1 & 0\\
                       \\
             & & \BigZero & & & & & \BigZero & & \\
                       &
       \end{pmatrix}. \]
    Hence, under the coordinate system defined by this basis, $q$ is represented by a quadratic polynomial of the form \[
        q=\sum^l_{i=1}(a_ix_i^2+x_iy_i+b_iy_i^2)+\sum^m_{i=l+1}a_ix_i^2. \]
    Since $k$ is algebraically closed, one can factor \[
        a_ix_i^2+x_iy_i+b_iy_i^2=(a_i'x_i+b_i'y_i)(a_i''x_i+b_i'y_i) \]
    into two linear forms, which are necessarily linearly independent since their product contains the term $x_iy_i$. For each $i=1,\ldots,l$, let \[
     \tilde{x}_i=a_i'x_i+b_i'y_i,\quad \tilde{y}_i=a_i''x_i+b_i''y_i.\]
    Then \[
    q=  \begin{cases}
            \ \displaystyle\sum_{i=1}^l\tilde{x}_i\tilde{y}_i, &\text{if $a_i=0$ for all }i>l; \\
            \ \displaystyle\sum_{i=1}^l\tilde{x}_i\tilde{y}_i+\tilde{x}_{l+1}^2, & \text{if } \displaystyle\tilde{x}_{l+1}=\sum^m_{i=l+1}\sqrt {a_i}x_i\ne 0.
        \end{cases}\]
\end{proof}

\begin{corollary}\label{cor: irred curve in P2 in char 2}
    Let $C\subseteq \PP^2$ be an irreducible quadric. Then there exists a homogeneous coordinate system $[x_0:x_1:x_2]$ for $\PP^2$ such that the defining equation of $C$ is $x_0x_1+x_2^2=0$.
\end{corollary}

\begin{corollary}\label{cor: tangent of conic in char 2 }
    Let $C\subseteq \PP^2$ be an irreducible quadric.
    \begin{enumerate}
        \item There exists a unique point $p\in \PP^2$ such that every line through $p$ is tangent to $C$, and every line tangent to $C$ passes through $p$.
        \item For any point $q\in \PP^2$ distinct from $p$, there is only one line through $q$ tangent to $C$.
    \end{enumerate}
\end{corollary}

\begin{proof}
    By Corollary \ref{cor: irred curve in P2 in char 2}, we may assume that the defining polynomial of $C$ is \[
     f(x_0,x_1,x_2)=x_0x_1+x_2^2. \]
    Since \[
        (\dfrac{\partial f}{\partial x_0},\dfrac{\partial f}{\partial x_1},\dfrac{\partial f}{\partial x_2})=(x_1,x_0,0), \]
    the tangent line of $C$ at a point $[a_0:a_1:a_2]\in C$ is \[
            a_1x_0+a_0x_1=0. \]
    These lines all pass through the point $p=[0:0:1]$, and conversely, every line through $p$ can be obtained in this way. Thus the first statement is proved.

    Let $q\in\PP^2$ be a point distinct from $p$, and let $L$ be a line through $q$ tangent to $C$. Then $L$ passes through $p$ by the first statement. Hence $L$ is the unique line passing through $p$ and $q$, so the second statement is proved.
\end{proof}

\end{document}